\documentclass{amsart}
\usepackage{amsmath,amsthm,amsfonts, amssymb,amscd}
\usepackage[all]{xy}

\newtheorem{theorem}{Theorem}[section]
\newtheorem{lemma}[theorem]{Lemma}
\newtheorem{example}[theorem]{Example}
\newtheorem{proposition}[theorem]{Proposition}

\newcommand{\dx}{\, \mbox{\rm d}}

\newcommand{\Sh}{\mbox{\rm Sh}}

\begin{document}
\title[Olson Order of Quantum Observables]{Olson Order of Quantum Observables}
\author[Anatolij Dvure\v{c}enskij]{Anatolij Dvure\v{c}enskij$^{1,2}$}
\date{}%
\maketitle

\begin{center}  \footnote{


The paper has been supported by the Slovak Research and Development Agency under the contract APVV-0178-11, the grant VEGA No. 2/0069/16 SAV
 and GA\v{C}R 15-15286S. }
Mathematical Institute,  Slovak Academy of Sciences\\
\v Stef\'anikova 49, SK-814 73 Bratislava, Slovakia\\
$^2$ Depart. Algebra  Geom.,  Palack\'{y} Univer.\\
17. listopadu 12, CZ-771 46 Olomouc, Czech Republic\\

E-mail: {\tt
dvurecen@mat.savba.sk}
\end{center}

\begin{abstract}
Using ideas of Olson \cite{Ols} who showed that the system of effect operators of a Hilbert space can be ordered by the so-called spectral order such that the system of effect operators is a complete lattice. Using his ideas, we introduce a partial order, called the Olson order, on the set of bounded observables of a complete lattice effect algebra. We show that the set of bounded observables is a Dedekind complete lattice.
\end{abstract}

\section{Introduction}

Quantum mechanics is a very effective way for description of the physical world. In the last decades new discoveries have been found for its applications to quantum information and quantum computing. In a classical physical system, the measurable events form a Boolean algebra. However, in the quantum mechanical world, this is not a case. Therefore, Birkhoff and von Neumann \cite{BiNe} introduced orthomodular lattices as the event structure describing quantum mechanical experiments. Later, orthomodular lattices and orthomodular posets were considered as the standard quantum logics \cite{Var}. In the nineties of the last century, two equivalent quantum structures, D-posets, \cite{KoCh},  and effect algebras, \cite{FoBe}, were introduced. They generalize many known quantum structures like Boolean algebras, orthomodular posets, MV-algebras, orthoalgebras, etc.

Effect algebras are partial algebraic systems with the primary notion $+$ such that $a+b$ denotes the disjunction of two mutually excluding events. For example, for sets $F$ and $G$, $F+G$ means $F$ and $G$ are disjoint and $F+G = F\cup G$.

The most important example of effect algebras, which is also crucial for the so-called Hilbert space quantum mechanics, is the set $\mathcal E(H)$ of all Hermitian operators on a real, complex or quaternionic Hilbert space $H$ which are between the zero and the identity operators. It is an interval in the po-group $\mathcal B(H)$ of Hermitian operators under the standard ordering of operators: $A\le B$ iff $(A\phi,\phi)\le (B\phi,\phi)$ for each unit vector $\phi \in H$. The second quantum structure  important for quantum mechanics is the system $\mathcal P(H)$ of all orthogonal projectors on $H$, or equivalently, the system $\mathcal L(H)$ of all closed subspaces of $H$. From the ordering point of view, using the standard order of operators,  $\mathcal P(H)\cong \mathcal L(H)$ is a complete orthomodular poset, \cite{Var}, $\mathcal E(H)$ is not a lattice, and $\mathcal B(H)$ is an antilattice, \cite{Kad}, i.e. only comparable operators have $\vee$ and $\wedge$.

Olson \cite{Ols} introduced in 1971 a so-called spectral order of  Hermitian operators which is defined as follows: If $A,B$ are two Hermitian operators with the spectral measures $E_A$ and $E_B$, then $A\preceq_s B$ iff $E_B((-\infty,t])\le E_A((-\infty,t])$ for each $t \in \mathbb R$. With respect  to this order, $\mathcal E(H)$ is a complete lattice which is a sublattice of $\mathcal B(H)$ with respect to the spectral order, the latter is a Dedekind complete lattice. We note that the spectral order was re-introduced in \cite{dGr} in 2005,  and in the last period, it is studied in some papers in more details, for example, in \cite{GLP} some interesting properties of $\mathcal E(H)$ are established.

Gudder in \cite{Gud} introduced a so-called logical order $\preceq$ on $\mathcal B(H)$ by $A\preceq B$ iff $AB = A^2$. Under this order, $\mathcal B(H)$ is a near lattice, i.e. $A \wedge B$ and $A\vee B$ exist only if there $C$ such that $A,B \preceq C$.

A crucial notion of the theory of effect algebras is an observable which models quantum measurement and which in the case of $\mathcal P(H)$ is equivalent to spectral measures and they are in a one-to-one description with self-adjoint operators. In addition, observables for $\mathcal E(H)$ are POV-measures.

Using ideas of Olson, \cite{Ols}, we introduce a spectral order called also the Olson order on the set of all observables of any monotone $\sigma$-complete effect algebra, and we show that the set of all bounded observables is with respect to the Olson order a Dedekind complete lattice and a Dedekind $\sigma$-lattice for complete lattice effect algebras and $\sigma$-lattice effect algebras, respectively. For the Olson order we have that for question observables $q_a\preceq_s q_b$ iff $a\le b$. In particular, we show that the set of observables whose spectrum is in the real interval $[0,1]$ forms a structure similar to the structure of $\mathcal E(H)$ under the spectral order.

The paper is organized as follows. Section 2 gathers the basic notions of the theory of effect algebras. Section 3 is the main part of the paper, where the Olson order of observables is introduced, and the lattice properties of the set of bounded observables of a complete lattice effect algebra are established. An equivalent definition of the spectral order with applications in $\mathcal P(H)$ and $\mathcal E(H)$ is presented in Section 4. Finally, Section 5 deals with spectral orders in some particular cases of effect algebras as $\sigma$-algebras, Boolean $\sigma$-algebras, tribes and effect tribes, where the latter two structures are ones of $[0,1]$-valued functions; here the lattice operations, $+$, and the order of functions are defined by points.

\section{Elements of Effect Algebras}

We remind that according to \cite{FoBe}, an {\it effect algebra} is  a partial algebra $E =
(E;+,0,1)$ with a partially defined operation $+$ and with two constant
elements $0$ and $1$  such that, for all $a,b,c \in E$, we have
\begin{enumerate}

\item[(i)] $a+b$ is defined in $E$ if and only if $b+a$ is defined, and in
such a case $a+b = b+a;$

 \item[(ii)] $a+b$ and $(a+b)+c$ are defined if and
only if $b+c$ and $a+(b+c)$ are defined, and in such a case $(a+b)+c
= a+(b+c);$

 \item[(iii)] for any $a \in E$, there exists a unique
element $a' \in E$ such that $a+a'=1;$

 \item[(iv)] if $a+1$ is defined in $E$, then $a=0.$
\end{enumerate}

If we define $a \le b$ if and only if there exists an element $c \in
E$ such that $a+c = b$, then $\le$ is a partial ordering on $E$, and
we write $c:=b-a.$ It is clear that $a' = 1 - a$ for any $a \in E.$ We note that an effect algebra is not necessarily a lattice.  For more information about effect algebras we can recommend the monograph \cite{DvPu}.

There are two basic sources of examples of effect algebras:

(1) Let $E$ be a system of fuzzy sets on $\Omega,$ that is $E \subseteq [0,1]^\Omega,$ such that
(i) $1 \in E$, (ii) $f \in E$ implies $1-f \in E$, and (iii) if $f,g
\in E$ and $f(\omega) \le 1 -g(\omega)$ for any $\omega \in \Omega$,
then $f+g \in E$. Then $E$ is an effect algebra of fuzzy sets which
is not necessarily a Boolean algebra.

(2) If $G$ is an Abelian partially ordered group written additively, $u \in G^+$, then $\Gamma(G,u):=[0,u]=\{g \in G: 0 \le g \le u\}$ is an effect algebra with $0=0,$ $1=u$ and $+$ is the group addition of elements if it exists in $\Gamma(G,u).$ In particular, if $G =\mathbb R,$ the group of real numbers, then $[0,1]=\Gamma(\mathbb R,1)$ is the standard effect algebra of the real interval $[0,1];$ it is an interval effect algebra. Here we can assign also the effect algebra $\mathcal E(H)$ because $\mathcal E(H)=\Gamma(\mathcal B(H),I)$, where $\mathcal B(H)$ is the set of Hermitian operators and $I$ is the identity operator.

We say that an effect algebra $E$ satisfies the Riesz Decomposition Property (RDP for short) if, for all $a_1,a_2,b_1,b_2 \in E$ such that $a_1 + a_2 = b_1+b_2,$ there are four elements $c_{11},c_{12},c_{21},c_{22}$ such that $a_1 = c_{11}+c_{12}$, $a_2= c_{21}+c_{22},$ $b_1= c_{11} + c_{21}$ and $b_2= c_{12}+c_{22}$. We note that if an effect algebra $E$ satisfies RDP, there is a unique po-groups $G$ satisfying RDP with a fixed strong unit (i.e. given $g\in G$, there
is an integer $n\ge 1$ such that $g \le nu$) such that $E \cong \Gamma(G,u)$, see \cite{Rav}.

A mapping $h: E\to F$ is said to be a {\it homomorphism} of effect algebras $E$ and $F$ such that (i) $h(1)=1$, (ii) if $a+b$ exists in $E$, then $h(a)+h(b)$ exists in $F$ and $h(a+b)=h(a)+h(b)$. If $h$ is bijective such that both $h$ and $h^{-1}$ are homomorphisms, then $h$ is said to be an {\it isomorphism}.

We define $\sum_{i=1}^n a_i:= a_1+\cdots +a_n$, if the element on the right-hand exists in $E.$ A system of elements $\{a_i: i \in I\}$ is said to be {\it summable} if, for any finite set $F$ of $I,$ the element $a_F:= \sum_{i\in F} a_i$ is defined in $E.$  If there is an element $a:= \sup \{a_F:  F$ is a finite subset of $I\},$ we call it the {\it sum} of $\{a_i: i \in I\}$ and we write $a = \sum_{i \in I}a_i$.

An effect algebra $E$ is {\it monotone} $\sigma$-{\it complete} if, for any sequence $a_1 \le a_2\le \cdots,$ the element $a = \bigvee_n a_n$  is defined in $E$ (we write $\{a_n\}\nearrow a$). Equivalently, every summable sequence has a sum. If an effect algebra is a lattice or a $\sigma$-lattice or a complete lattice, we say that $E$ is a {\it lattice effect algebra}, a $\sigma$-{\it lattice effect algebra}, and a {\it complete lattice effect algebra}, respectively.

If $E$ and $F$ are two monotone $\sigma$-complete effect algebras, a homomorphism $h:E \to F$ is said to be a $\sigma$-{\it homomorphism} if $\{a_n\} \nearrow a$ implies $\{h(a_n)\} \nearrow h(a)$ for $a, a_1,\ldots \in E$.

An {\it effect-tribe}  is any system ${\mathcal T}$ of fuzzy sets on
$\Omega\ne \emptyset $ such that (i) $1 \in {\mathcal T}$, (ii) if $f
\in {\mathcal T},$ then $1-f \in {\mathcal T}$, (iii) if $f,g \in {\mathcal T}$,
$f \le 1-g$, then $f+g \in {\mathcal T},$ and (iv) for any sequence
$\{f_n\}$ of elements of ${\mathcal T}$ such that $f_n \nearrow f$
(pointwise), then $f \in {\mathcal T}$. It is evident that any
effect-tribe is a monotone $\sigma$-complete effect algebra.

We note that $\mathcal E(H)$ is isomorphic to some effect-tribe. Indeed, let $\Omega:=\{\phi\in H: \|\phi\| = 1\}$, and given $A \in \mathcal E(H)$, let $f_A:\Omega\to [0,1]$ such that $f_A(\phi)=(A\phi,\phi)$, $\phi \in \Omega$. Then $\mathcal T(H):=\{f_A: A \in\mathcal E(H)\}$ is an effect-tribe $\sigma$-isomorphic to $\mathcal E(H)$.

An element of an effect algebra $E$ is said to be {\it sharp} if $a \wedge a'$ exists in $E$ and $a\wedge a'=0.$ Let $\Sh(E)$ be the set of sharp elements of $E$. Then (i) $0,1\in \Sh(E),$ (ii) if $a \in \Sh(E),$ then $a'\in \Sh(E).$ If $E$ is a lattice effect algebra, then $\Sh(E)$ is an orthomodular lattice which is a subalgebra and a sublattice of $E$,  \cite{JeRi}. If an effect algebra $E$ satisfies RDP, then by \cite[Thm 3.2]{Dvu2}, $\Sh(E)$ is even a Boolean algebra.

A very important family of effect algebras is the family of
MV-algebras, which were introduced by Chang \cite{Cha}.

We recall that an MV-algebra is an algebra $M = (M;\oplus, ^*,0,1)$
of type (2,1,0,0) such that, for all $a,b,c \in M$, we have

\begin{enumerate}
\item[(i)]  $a \oplus  b = b \oplus a$;
\item[(ii)] $(a\oplus b)\oplus c = a \oplus (b \oplus c)$;
\item[(iii)] $a\oplus 0 = a;$
\item[(iv)] $a\oplus 1= 1;$
\item[(v)] $(a^*)^* = a;$
\item[(vi)] $a\oplus a^* =1;$
\item[(vii)] $0^* = 1;$
\item[(viii)] $(a^*\oplus b)^*\oplus b=(a\oplus b^*)^*\oplus a.$
\end{enumerate}

If we define a partial operation $+$ on $M$ in such a way that $a+b$
is defined in $M$ if and only if $a \le b^*$ and  we set
$a+b:=a\oplus b$, then $(M;+,0,1)$ is an effect algebra with RDP.
For example, if $G$ is an Abelian lattice ordered group and $u \ge0,$ then
$(\Gamma(G,u); \oplus, ^*, 0,u)$, where  $\Gamma(G,u):= \{g\in G: 0\le g \le u\},$ $a\oplus b :=(a+b) \wedge u$, and $a^* := u-a$ $(a,b \in \Gamma(G,u))$ is an MV algebra, and every MV algebra arises in this way. We recall that every MV-algebra is a distributive lattice. In particular, if we take $\mathbb R,$ the $\ell$-group of real numbers, then $[0,1]= (\Gamma(\mathbb R,1); \oplus, ^*, 0,1)$ is the standard MV-algebra of the real interval $[0,1].$

We recall that a {\it tribe} on $\Omega \ne \emptyset$
is a collection ${\mathcal T}$ of fuzzy sets from $[0,1]^\Omega$ such
that (i) $1 \in {\mathcal T}$, (ii) if $f \in {\mathcal T}$, then $1 - f \in
{\mathcal T},$ and (iii) if $\{f_n\}$ is a sequence from ${\mathcal T}$,
then $\min \{\sum_{n=1}^\infty f_n,1 \}\in {\mathcal T}$.  A tribe is
always a $\sigma$-complete MV-algebra, where all operations are defined by points. We note that a tribe is a generalization of a $\sigma$-algebra of subsets of $\Omega$, because if $f_n = \chi_{A_n},$ where $\chi_A$ is the characteristic function of the set $A,$  then  $\min \{\sum_{n=1}^\infty \chi_{A_n},1 \} = \chi_{\bigcup_n A_n}.$

We note that every tribe is an effect-tribe, but the converse is not true, in general.

The notions of an effect-tribe and of a tribe are important for theory of effect algebras because every monotone $\sigma$-complete effect algebra with RDP is a $\sigma$-homomorphic image of some effect-tribe with RDP, see \cite{BCD}, and every $\sigma$-complete effect algebra is a $\sigma$-homomorphic image of some tribe, see \cite{Dvu1, Mun}.

\section{Observables and Olson's Order of Observables}

In this section, we introduce observables as models of quantum measurements, and we introduce the spectral order, we call it also the Olson order, of observables, and we exhibit the lattice properties of the set of all bounded observables.

Let $E$ be a monotone $\sigma$-complete effect algebra. An {\it observable} on $E$ is any mapping $x:\mathcal B(\mathbb R)\to E$, where $\mathcal B(\mathbb R)$ is the Borel $\sigma$-algebra of the real line $\mathbb R$, such that (i) $x(\mathbb R)=1$, (ii) if $E,F \in \mathcal B(\mathbb R)$, $E \cap F= \emptyset$, then $x(E\cup F)=x(E)+x(F)$, and (iii) if $\{E_i\}$ is a sequence of Borel sets such that $E_i\subseteq E_{i+1}$ for each $i$ and $\bigcup_i E_i=E$, then $x(E) = \bigvee_i x(E_i)$.

In other words, any observable is a $\sigma$-homomorphism of monotone $\sigma$-complete effect algebras. The basic properties of observables are
(i) $x(\mathbb R \setminus E)=x(E)'$, (ii) $x(\emptyset)=0$, (iii) $x(E)\le x(F)$ whenever $E \subseteq F$, and $x(F\setminus E)= x(F)-x(E)$, (iv) if $\{E_i\} \searrow E$, then $x(E) = \bigwedge_i x(E_i)$.

We denote by $\mathcal O(E)$ the set of observables on $E$. If $f: \mathbb R\to \mathbb R$ is a Borel measurable function, then the mapping $f(x): \mathcal B(\mathbb R) \to E$, defined by $f(x)(E):=x(f^{-1}(E))$, $E \in \mathcal B(\mathbb R)$, is an observable on $E$. For example, $x^2$ denotes $f(x)$, where $f(t)=t^2$.

Observables can be constructed also as follows. Let $\{a_n\}$ be a finite or infinite sequence of summable elements, $\sum_n a_n = 1$, and let $\{t_n\}$ be a sequence of mutually different real numbers. Then the mapping $x: \mathcal B(\mathbb R) \to E$ defined by

$$
x(E):= \sum\{a_n: t_n \in E\}, \ E \in \mathcal B(\mathbb R), \eqno(3.0)
$$
is an observable on $E$. In particular, if $t_0=0$, $t_1=1$ and $a_0=a'$ $a_1=a$ for some fixed element $a \in E,$  $x$ defined by (3.0) is an observable, called a {\it question} corresponding to the element $a$, and we write $x=q_a$.

An observable $x$ is bounded if there is a compact set $C$ such that $x(C)=1$. The least closed subset $K$ of $\mathcal B(\mathbb R)$ is said to be a {\it spectrum} of $x$, and we denote it by $\sigma(x)$; since the natural topology of the real line satisfies the second countability axiom, $\sigma(x)$ exists, and $x(\sigma(x))=1$. An observable is a question iff $\sigma(x)\subseteq \{0,1\}$. We note that a point $\lambda \in \mathbb R$ belongs to $\sigma(x)$ iff for any open set $U$ containing $\lambda$, $x(U)\ne 0$.

The following representation theorem of observables was originally proved for $\sigma$-orthocomplete orthomodular posets with an order-determining system of $\sigma$-additive states in \cite{Cat}. Then it was generalized for observables on $\sigma$-lattice effect algebras in \cite[Thm 3.5]{DvKu}, as well as for $\sigma$-complete MV-algebras \cite[Thm 3.2]{DvKu}, for monotone $\sigma$-complete effect algebras with RDP \cite[Thm 3.9]{DvKu}, and for some monotone $\sigma$-complete effect algebras in \cite{270}.

\begin{theorem}\label{th:2.2}
Let $x$ be an observable on a $\sigma$-lattice effect algebra $E$. Given a real number $t \in \mathbb R,$ we put

$$ x_t := x((-\infty, t)). \eqno(3.1)
$$
Then

$$ x_t \le x_s \quad {\rm if} \ t < s, \eqno (3.2)$$

$$\bigwedge_t x_t = 0,\quad \bigvee_t x_t =1, \eqno(3.3)
$$
and
$$ \bigvee_{t<s}x_t = x_s, \ s \in \mathbb R. \eqno(3.4)
$$

Conversely, if there is a system $\{x_t: t \in \mathbb R\}$ of elements of $E$ satisfying {\rm
(3.2)--(3.4)}, then there is a unique observable $x$ on $E$ for which $(3.1)$ holds for any $t \in \mathbb R.$
\end{theorem}

We note that Theorem \ref{th:2.2} holds also if we change the system $\{x_t: t \in \mathbb R\}$ to a system $\{x_t: t \in \mathbb Q\}$ satisfying (3.2)--(3.4), where $\mathbb Q$ is the set of all rational numbers.

The system $\{x_t: t \in \mathbb R\}$ from Theorem \ref{th:2.2} satisfying (3.2)--(3.4) is said to be the {\it spectral resolution} of an observable $x$.

For example, if $x=q_a$ for some element $a \in E$, then the spectral resolution for $q_a$ on a monotone $\sigma$-complete effect algebra $E$ is as follows

$$
q_a((-\infty, t))= \left\{\begin{array}{ll} 0 & \mbox{if} \ t\le 0,\\
a' & \mbox{if}\ 0< t\le 1,\\
1 & \mbox{if}\ 1<t
\end{array}
\right.
\eqno(3.5)
$$
for $t \in \mathbb R$.

The spectral resolution was used in \cite{Ols} to define a new order $\preceq_s$ for the effect algebra $\mathcal E(H)$ ($s$ stands for spectral) in order to be $\mathcal E(H)$ a complete lattice such that for orthogonal projections it coincides with the standard order of operators. This new order was re-introduced for $\mathcal E(H)$ in \cite{dGr}.

Inspiring by \cite{Ols,dGr}, we write that two observables $x$ and $y$ on a monotone $\sigma$-complete effect algebra $E$ are in a relation
$$
x \preceq_s y \quad \Leftrightarrow \quad  y((-\infty,t))\le x((-\infty, t)) \ \mbox{for each} \ t \in \mathbb R,
\eqno(3.6)
$$
and the relation $\preceq_s$ is a partial order on $\mathcal O(E)$. We will call the order $\preceq_s$ also the {\it Olson order} or the     {\it spectral order}.

From Theorem \ref{th:2.2} we have that $\preceq_s$ is a partial order on the set of all observables $\mathcal O(E)$ of a $\sigma$-lattice effect algebra $E$. We denote by $\mathcal{EO}(E)$ the class of observables $x$ such that $\sigma(x)\subseteq [0,1]$, and we denote by $\mathcal O_b(E)$ the class of bounded observables on $E$.

\begin{proposition}\label{pr:2.2}
Let $a$ and $b$ be elements of a monotone $\sigma$-lattice effect algebra $E$. Then $q_a\preceq_s q_b$ if and only if $a\le b$. In addition, for every $x\in \mathcal{EO}(E)$, we have $q_0 \preceq_s x\preceq_s q_1$.
\end{proposition}

\begin{proof}
The statement follows from (3.4). Moreoverf, $q_0((-\infty,t))=0$ if $t \le 0$ otherwise $q_0((-\infty,t))=1$, and $q_1((-\infty,t))= 0$ if $t\le 1$ otherwise $q_1((-\infty,t))=1$.
\end{proof}

If $x \in \mathcal{O}_b(E)$, then for its spectral resolution $\{x_t: t \in \mathbb R\}$, there are $s_0,t_0$ such that $x_t = 0$ for $t \le s_0$ and $x_t = 1$ for $t\ge t_0$.

For complete lattice effect algebras we are going to establish lattice properties of $\mathcal O_b(E)$ as follows.

\begin{lemma}\label{le:2.3}
Let $\{x_\alpha: \alpha \in A\}$ be a system of bounded observables on a complete lattice effect algebra $E$ such that there is a bounded observable $y$ on $E$ which is a lower bound of $\{x_\alpha: \alpha \in A\}$. Define
$$
x(t):= \bigvee_\alpha x_\alpha((-\infty,t)),\ t \in \mathbb R. \eqno(3.7)
$$
Then the system $\{x(t): t \in \mathbb R\}$ satisfies {\rm (3.2)--(3.4)} and it determines a unique bounded observable $x$ on $E$ and this observable is the greatest lower bound of $\{x_\alpha: \alpha \in A\}$ under the Olson order $\preceq_s$ in $\mathcal O_b(E)$, and we write $x = \bigwedge_\alpha x_\alpha$.
\end{lemma}

\begin{proof}
Assume $y \preceq_s x_\alpha$, $\alpha \in A$. We set $x_\alpha(t)=x_\alpha((-\infty,t))$ and $y(t)=y((-\infty,t))$. Then $x_\alpha(t)\le y(t)$ for each $t \in \mathbb R$ and each $\alpha \in A$. Define $x(t)$ by (3.7). We verify (3.2)--(3.4). Trivially $x(t) \le x(s)$ if $t\le s$. We have (i) $\bigvee_t \bigvee_\alpha x_\alpha(t)= \bigvee_\alpha \bigvee_t x_\alpha(t) =1$, (ii) $0\le \bigwedge_t \bigvee_\alpha x_\alpha(t) \le \bigwedge_t y(t)=0$, and (iii) $\bigvee_{t<s} x_t= \bigvee_{t<s}\bigvee_\alpha x_\alpha(t) = \bigvee_\alpha \bigvee_t x_\alpha(t)= \bigvee_\alpha x_\alpha(s)=x(s)$.

Let $x$ be a unique observable of $E$ for which $\{x(t): t \in \mathbb R\}$ is the spectral resolution. From the construction of $x$ we have $x(t)\ge x_\alpha(t)$ for each $\alpha\in A$ and each $t\in \mathbb R$. This implies $x\preceq_s x_\alpha$ for each $\alpha \in A$. Since $x(t) \le y(t)$, there is $t_0\in \mathbb R$ such that $x(t)=1$ for each $t\ge t_0$, and $x(t)\ge x_\alpha(t)$, there is $s_0\in \mathbb R$ such that $x(t)=0$ for $t<s_0$, showing $x$ is a bounded observable.

Finally, let $z$ be a bounded observable on $E$ such that $z\preceq_s x_\alpha$ for each $\alpha \in A$. Then $x_\alpha(t)\le z(t)$ for each $\alpha\in A$, i.e. $x(t)\le z(t)$, $t \in \mathbb R$. Hence, $z\preceq_s x$.
\end{proof}

\begin{proposition}\label{pr:2.4}
Let $\{x_t: t \in \mathbb R\}$ be a system of elements of a $\sigma$-lattice effect algebra $E$ satisfying {\rm (3.2)--(3.3)}. Define
$$
x_l(t):=\bigvee_{u<t}x(u), \ t \in \mathbb R. \eqno(3.8)
$$
Then $x_l(t)$ is defined in $E$ for each $t$, and the system $\{x_l(t): t \in \mathbb R\}$ satisfies {\rm (3.2)--(3.4)}.

Similarly, if we define
$$
x_r(t):= \bigwedge_{u>t} x(u), \ t \in \mathbb R, \eqno(3.9)
$$
then $x_r(t)$ exists in $E$ for each $t$, and the system $\{x_r(t): t \in \mathbb R\}$ satisfies {\rm (3.2)--(3.3)}  and $\bigwedge_{t>s}x_r(t)=x_r(s)$, $s \in \mathbb R$.

In addition, if there are $s_0$, $t_0$ such that $x(t)=0$ if $t<s_0$ and $x(t)=1$, if $t>t_0$, then $x_l(t)=x_r(t)=0$ if $t<s_0$ and $x_r(t)=x_l(t)=1$ if $t>t_0$.
\end{proposition}

\begin{proof}
(1)
Since $E$ is a $\sigma$-lattice, the element $x_l'(t)=\bigvee\{ x(u):u<t, u \in \mathbb Q\}$ exists in $E$. Due to the density of $\mathbb Q$ in $\mathbb R$, for each $u\in \mathbb R$ with $u<t$, there are two rational numbers $p$ and $q$ such that $p<u<q<t$. Hence, using the monotonicity of $\{x(t): t \in \mathbb R\}$, $x_l(t)$ exists in $E$ and $x_l(t)=x'_l(t)$.

We have:

(i) $x_l(t)\le x_l(s)$ if $t<s$, (ii) $0\le \bigwedge_t x_l(t) = \bigwedge_t \bigvee_{u<t} x(u)\le \bigwedge_{t<0}\bigvee_{u<t}x(u) \le \bigwedge_{t<0}x(t)=0$,  $1\ge \bigvee_t x_l(t)= \bigvee_t\bigvee_{u<t} x(u) \ge \bigvee_{t>0} x(t/2)=1$. (iii) $\bigvee_{t<s}x_l(t)= \bigvee_{t<s}\bigvee_{u<t}x(u) = \bigvee_{u<t}\bigvee_{t<s} x(u)= \bigvee_{u<s}x(u)=x_l(s)$.

(2) In the same way as for (1), we prove that $x_r(t)$ is defined in $E$ for each $t$. Check:

(i) $x_r(t)\le x_r(s)$ if $t<s$, (ii) $0\le \bigwedge_t x_r(t) = \bigwedge_t \bigwedge_{u>t} x(u) \le \bigwedge_{t<0} x(2t)$, (ii) $1\ge \bigvee_t x_r(t) = \bigvee_t \bigwedge_{u>t} x(u) \ge \bigvee_{t>0} x(t/2)=1$. (iii) $\bigwedge_{t>s}x_r(t)=\bigwedge_{t>s}\bigwedge_{u>t}x(u) =
\bigwedge_{u>t}\bigwedge_{t>s} x(u) = \bigwedge_{u>s}x(u)=x_r(s)$.

The rest is clear from (3.8) and (3.9).
\end{proof}

We note that $\{x_l(t): t \in \mathbb R\}$ and $\{x_r(t): t \in \mathbb R\}$ defined by (3.8) and (3.9)
are said to be the {\it left-regularization} and {\it right-regularization}, respectively, of the system $\{x(t): t \in \mathbb R\}$ satisfying (3.2)--(3.3).

\begin{lemma}\label{le:2.5}
Let $\{x_\alpha: \alpha \in A\}$ be a system of bounded observables on a complete lattice effect algebra $E$ such that there is a bounded observable $y$ on $E$ which is an upper bound of $\{x_\alpha: \alpha \in A\}$. Define
$$
x(t):= \bigvee_{u<t} \bigwedge_\alpha x_\alpha((-\infty,u)),\ t \in \mathbb R. \eqno(3.10)
$$
Then the system $\{x(t): t \in \mathbb R\}$ satisfies {\rm (3.2)--(3.4)}  and it determines a unique bounded observable $x$ on $E$ and this observable is the least upper bound of $\{x_\alpha: \alpha \in A\}$ under the Olson order $\preceq_s$ in $\mathcal O_b(E)$, and we write $x = \bigvee_\alpha x_\alpha$.
\end{lemma}

\begin{proof}
We set $x_\alpha(t)=x_\alpha((-\infty,t))$ and $y(t)=y((-\infty,t))$.
Since $x_\alpha \preceq_s y$, we have $y(t)\le x_\alpha(t)$. We define
$$
x^0(t):= \bigwedge_\alpha x_\alpha(t),\ t \in \mathbb R.
$$

We verify that $\{x^0(t): t \in \mathbb R\}$ satisfies conditions (3.2) and (3.3). (i) Clearly $x^0(s)\le x^0(t)$ if $s<t$. (ii) Check
$\bigwedge_t x^0(t) = \bigwedge_t \bigwedge_\alpha x_\alpha(t) = \bigwedge_\alpha \bigwedge_t x_\alpha(t)=0$. (iii) $1\ge \bigvee_t x^0(t) = \bigvee_t \bigwedge_\alpha x_\alpha(t) \ge \bigvee_t y(t)= 1$.

Let $\{x(t): t \in \mathbb R\}$ be the the left-regularization of $\{x^0(t): t \in \mathbb R\}$, i.e. $x(t)=x^0_l(t)$, $t \in \mathbb R$. By Proposition \ref{pr:2.4}, the system $\{x(t): t \in \mathbb R\}$ satisfies (3.2)--(3.4), so it defines a unique bounded observable $x$ on $E$ such that $x((-\infty,t))=x(t)$, $t \in \mathbb R$. Then $x(t)= \bigvee_{u<t}\bigwedge_\alpha x_\alpha(u) \le \bigvee_{u<t} x_\alpha(u) = x_\alpha(t)$ for each $\alpha \in A$, i.e. $x_\alpha \preceq_s x$ for each $\alpha \in A$. Now let $z$ be a bounded observable on $E$ such that $x_\alpha \preceq_s z$, $\alpha \in A$. Whence $z(u) \le x_\alpha(u)$, $u \in \mathbb R$, i.e. $z(t)= \bigvee_{u<t}z(u) \le \bigvee_{u<t} \bigwedge_\alpha x_\alpha(u) = x(t)$, $t \in \mathbb R$, which yields $x\preceq_ s  z$. Finally, $x = \bigvee_\alpha x_\alpha$.
\end{proof}

We say that a poset $(P; \leqslant )$ is {\it Dedekind complete} if (i) for each family $\{x_\alpha: \alpha \in A\}$ of elements of $P$ which is bounded from below, there is $\bigwedge_\alpha x_\alpha$ in $P$ and  for each family $\{y_\beta: \beta \in B\}$ of elements of $P$ which is bounded from above, there is $\bigvee_\beta y_\beta$ in $P$. In an analogous way we introduce a Dedekind $\sigma$-complete lattice, if instead of the arbitrary set $A$, we take only a countable set $A$.

\begin{theorem}\label{th:2.6}
The set $\mathcal O_b(E)$ of bounded observables of a complete lattice effect algebra $E$ is a Dedekind complete lattice under the Olson order.

The set $\mathcal{EO}(E)$ of observables whose spectra are in the interval $[0,1]$ is a complete lattice under the Olson order, and $q_0$ and $q_1$ are the bottom and top elements of $\mathcal {EO}(E)$.
\end{theorem}

\begin{proof}
The first statement follows from Lemma \ref{le:2.3}, where $x =\bigwedge_\alpha x_\alpha$ is defined by (3.7), and from Lemma \ref{le:2.5}, where $x=\bigvee_\alpha x_\alpha$ is defined by (3.10).

The second statement on $\mathcal{EO}(E)$ follows from the first part and Proposition \ref{pr:2.2}.
\end{proof}

If $E$ is a $\sigma$-complete effect algebra, it could happen that the elements (3.7) and (3.10) do not exist in $E$ if $A$ is not countable. However, for a countable set $A$, we can literally repeat formulas (3.7) and (3.10), so we have the following result.

\begin{theorem}\label{th:2.7}
The set $\mathcal O_b(E)$ of bounded observables of a $\sigma$-lattice effect algebra $E$ is a Dedekind $\sigma$-complete lattice under the Olson order.

The set $\mathcal{EO}(E)$ of observables whose spectra are in the interval $[0,1]$ is a complete lattice under the Olson order, and $q_0$ and $q_1$ are the bottom and top elements of $\mathcal {EO}(E)$.
\end{theorem}

Because the Olson order was defined also for non-lattice effect algebras, we note that it could happen, that $x\wedge y$ exists for a non-lattice effect algebra. In what follows we show such a situation.

We remind that an observable $x$ is {\it simple} if $\sigma(x)$ is a finite non-empty set. If $t_1<\cdots <t_n$ and $a_i=x(\{t_i\})$, $i=1,\ldots,n$, then $a_1+\cdots+a_n=1$, and for the spectral resolution of $x$, we have

$$
x((-\infty, t))= \left\{\begin{array}{ll} 0 & \mbox{if} \ t\le t_1,\\
a_1+\cdots + a_i & \mbox{if}\ t_i< t\le t_{i+1},\ i=1,\ldots,n-1,\\
1 & \mbox{if}\ t_n<t
\end{array}
\right.
\eqno(3.11)
$$
for $t \in \mathbb R$, and $\sigma(x)\subseteq \{t_1,\ldots,t_n\}$. Every question observable is a simple observable.

We note that if $x$ and $y$ are two simple observables, then there are a finite set $\{t_1,\ldots,t_n\}$ of real numbers with $t_1<\cdots< t_n$, two finite sequences of elements of $E$ $a_1,\ldots,a_n$ and $b_1,\ldots, b_n$ with $a_1+\cdots + a_n=1=b_1+\cdots +b_n$  such that $x(\{t_i\})=a_i$ and $y(\{t_i\})=b_i$ for each $i=1,\ldots,n$. Indeed, if $\sigma(x)=\{u_1,\ldots,u_k\}$ and $\sigma(y)=\{v_1,\ldots, u_l\}$, then $\sigma(x)\cup \sigma(y)= \{t_1,\ldots, t_n\}$ for some $t_1<\dots<t_n$. If we set $a_i=x(\{t_i\})$ and $b_i=y(\{t_i\})$ for $1\le i \le n$, then $a_1+\cdots +a_n=1=b_1+\cdots + b_n$, and $\sigma(x),\sigma(y)\subseteq \{t_1,\ldots,t_n\}$.

\begin{example}\label{ex:2.8} {\rm (1)} Let $x$ and $y$ be two simple observables on a monotone $\sigma$-complete effect algebra $E$ such that there is a finite sets of real numbers $t_1<t_2<\cdots <t_n$, such that $\sigma(x), \sigma(y)\subseteq\{t_1,\ldots,t_n\}$. Set $a_i = x(\{t_i\})$, $b_i=y(\{t_i\})$ for $i=1,\ldots, n$. Then $x\wedge y$ exists in $\mathcal O_b(E)$ if and only if
$d_i:=(a_1+\cdots +a_i)\vee (b_1+\cdots+b_i)$ exists in $E$ for each $i=1,\ldots,n$, and in such a case, $x\wedge y = z$, where $z(\{t_1\})= d_1$, $z(\{t_i\})= d_{i+1}-d_i$ for $i=1,\ldots,n-1$, and $z(\{t_n\})= 1-d_n$.

{\rm (2)} If $q_a$ and $q_b$ are two question observables on a monotone $\sigma$-complete effect algebra $E$, then $q_a\wedge q_b$ exists iff $a\wedge b$ exists in $E$, and in such a case, $q_a\wedge q_b=q_{a\wedge b}$.
\end{example}

\begin{proof}
(1) Let $z=x\wedge y$ is defined. Then $x(t), y(t)\le z(t)$ for each $t$. Define elements $c_i = z(t_{i+1})$, $c_i = z(t_{i+1}) -z(t_i)$ for $i=2,\ldots, n-1$ and $c_n = 1 -z(t_n)$. Then $c_1+\cdots +c_n=1$. Define an observable $u$ such that $u(\{t_i\})=c_i$, $i=1,\ldots, n$. Using (3.11), we see that $u\preceq_s x,y$, which gets $u\preceq_s z$. We set $a^i:=a_1+\cdots+a_i$ and $b^i:=b_1+\cdots +b_i$, $i=1,\ldots,n$. We assert that $u(t_{i+1})=z(t_{i+1})$ is the least upper bound for $a^i$ and $b^i$ for $i=1,\ldots,n-1$. Of course, $a^i,b^i \le u(t_{i+1})=c_1+\cdots +c_i$ for $i=1,\ldots,n-1$, and $a^n=1=b^n=a_1+\cdots+c_n$. Now let $d$ be an upper bound for $a^i,b^i$ where $i$ is a fixed index such that $i=1,\ldots,n-1$. Define an observable $v$ on $E$ such that $v(t)= 0$ if $t\le t_1$, $v(t)=d$ if $0< t \le t_{i}$, and $v(t)=1$ if $t>t_i$. Then $v\preceq_s x,y$ which yields $v \preceq_s z$, so that $u(t_{i+1})=z(t_{i+1})\le d$, which proves $a^i\vee b^i=z(t_{i+1})$ for $i=1,\ldots, n-1$.

Suppose the converse, i.e. $d^i:=a^i\vee b^i$ exists in $E$ for each $i=1,\ldots,n$. Define an observable $z$ such that $z(t)=0$ if $t\le t_1$, $z(t)=d^i$ if $t_i<t\le t_{i+1}$, $i=1,\ldots, n-1$, and $z(t)=1$ if $t>t_n$. Then $v\preceq_s x,y$. Assume that $z$ is an arbitrary bounded observable on $E$ such that $z \preceq_s x,y$. Then $z(t)=0$, $z(t)\ge d^i$ if $t_i<t\le t_{i+1}$ and $z(t)=1$. Hence, $v(t)\le z(t)$ for each $t$, so that $z\preceq_s v$, which proves $v=x\wedge y$.

(2) It follows directly from (1).
\end{proof}

\section{An Equivalent Approach to the Olson Order and Hermitian Operators}

We note that Olson \cite{Ols} and later de Groote \cite{dGr} have defined the spectral order on the set $\mathcal E(H)$ in a little bit different way as we did. In what follows, we show that both approaches are the same. First we establish an analogue of Theorem \ref{th:2.2} which is important for a new approach.

\begin{theorem}\label{th:3.1}
Let $x$ be an observable on a $\sigma$-lattice effect algebra $E$. Given a real number $t \in \mathbb R,$ we put

$$ x_t := x((-\infty, t]). \eqno(4.1)
$$
Then

$$ x_t \le x_s \quad {\rm if} \ t < s, \eqno (4.2)$$

$$\bigwedge_t x_t = 0,\quad \bigvee_t x_t =1, \eqno(4.3)
$$
and
$$ \bigwedge_{s<t}x_t = x_s, \ s \in \mathbb R. \eqno(4.4)
$$

Conversely, if there is a system $\{x_t: t \in \mathbb R\}$ of elements of $E$ satisfying {\rm
(4.2)--(4.4)}, then there is a unique observable $x$ on $E$ for which $(4.1)$ holds for any $t \in \mathbb R.$
\end{theorem}

\begin{proof}
Theorem can be proved in an analogous way as \cite[Thm 3.2, Thm 3.5]{DvKu}. We present here another proof.

One direction of the proof is clear. For the another one, assume that a system $\{x(t): t \in \mathbb R\}$ of elements of $E$ satisfies conditions (4.2)--(4.4). According to Proposition \ref{pr:2.4}, we define the left-regularization $\{x_l(t): t \in \mathbb R\}$ of the system $\{x(t): t \in \mathbb R\}$. Due to Proposition \ref{pr:2.4}, $\{x_l(t): t \in \mathbb R\}$ satisfies conditions (3.2)--(3.4), so there is a unique observable $x^0$ on $E$ such that $x^0((-\infty, t))=x_l(t)$, $t \in \mathbb R$. Take the right-regularization $\{x_r^0(t): t \in \mathbb R\}$ of $\{x^0(t): t \in \mathbb R\}$, where $x^0(t)=x^0((-\infty, t))$, $t \in \mathbb R$.

We have $x^0_r(t)= \bigwedge_{u>t} x^0((-\infty,t))= x^0((-\infty, t])$, $t \in \mathbb R$, as well as, $x^0_r(t)= \bigwedge_{u>t} x^0((-\infty,t)) = \bigwedge_{u>t}\bigvee_{v<u} x(v) \le \bigwedge_{u>t} x(u) = x(t)$, so that $x^0((-\infty,t])\le x(t)$.

On the other hand, $x_r^0(t)= \bigwedge_{u>t}\bigvee_{v<u} x(v)$. Take $v_0$ such that $t<v_0<u$, then $x_r^0(t)= \bigwedge_{u>t}\bigvee_{v<u} x(v)\ge \bigwedge_{u>t}x(v_0) =x(v_0)\ge x(t)$. This entails $x^0((-\infty,t])=x(t)$, $t \in \mathbb R$.

We assert that $x^0$ is a unique observable on $E$ such that $x^0((-\infty,t])=x(t)$, $t \in \mathbb R$. Indeed, let $y$ is any observable on $E$ such that $y((-\infty,t])=x_t$, $t \in \mathbb R$. Let $\mathcal H$ be the set of Borel sets $E \in \mathcal B(\mathbb R)$ such that $x(E)=y(E)$. Then $\mathcal H$ contains all intervals of the form $(-\infty, t]$ for each $t\in \mathbb R$. Since $(-\infty,t) = \bigcup_{s_n\nearrow t}(-\infty, s_n]$, we have $(-\infty,t)\in \mathcal H$.
Then $\mathcal H$ is a Dynkin system, i.e. a system of subsets containing its universe which is closed under the set theoretical complements and countable unions of  disjoint subsets, \cite{Bau}. The system $\mathcal H$ contains all intervals $(-\infty, t)$ for $t \in \mathbb R$; these intervals form a $\pi$-system, i.e. intersection of two sets from the $\pi$-system is from the $\pi$-system.  Hence, by \cite[Thm 2.1.10]{Dvu} or \cite[Thm 1.1]{Kal}, $\mathcal H$ is also a $\sigma$-algebra, and finally we have $\mathcal H = \mathcal B(\mathbb R).$
\end{proof}

\begin{lemma}\label{le:3.2}
Let $x$ and $y$ be observables of a $\sigma$-lattice effect algebra $E$. Then $x\preceq_s y$ if and only if $y((-\infty,t])\le x((-\infty,t])$ for each $t \in \mathbb R$.
\end{lemma}

\begin{proof}
Let $x\preceq_s y$. Then $y((-\infty,s))\le x((-\infty,s))$, $s \in \mathbb R$ which gets $y((-\infty, t])=\bigwedge_{s>t}y((-\infty,s))\le \bigwedge_{s>t}x((-\infty,s)) = x((-\infty, t])$, $t \in \mathbb R$.

Conversely, assume $y((-\infty,s])\le x((-\infty,s])$, $s \in \mathbb R$. Then
$x((-\infty,t))=\bigvee_{s<t} x((-\infty,s]) \le \bigvee_{s<t}y((-\infty, t])= y((-\infty, t))$, $t \in \mathbb R$.
\end{proof}

In the next two results we show how we can calculate infima and suprema of a system of observables $\{x_\alpha: \alpha \in A\}$ using the systems $\{x_\alpha((-\infty,t]): t \in \mathbb R\}$, $\alpha \in A$, instead of the systems $\{x_\alpha((-\infty,t)): t \in \mathbb R\}$, $\alpha \in A$, as it was done in \cite{Ols, dGr} for Hermitian operators on a Hilbert space $H$.

\begin{lemma}\label{le:3.3}
Let $\{x_\alpha: \alpha \in A\}$ be a system of bounded observables on a complete lattice effect algebra $E$ such that there is a bounded observable $y$ on $E$ which is a lower bound of $\{x_\alpha: \alpha \in A\}$. Define
$$
x(t):=\bigwedge_{u>t} \bigvee_\alpha x_\alpha((-\infty,u]),\ t \in \mathbb R. \eqno(4.5)
$$
Then the system $\{x(t): t \in \mathbb R\}$ satisfies {\rm (4.2)--(4.4)} of Theorem {\rm \ref{th:3.1}} and it determines a unique bounded observable $x$ on $E$ and this observable is the greatest lower bound of $\{x_\alpha: \alpha \in A\}$ under the Olson order $\preceq_s$ in $\mathcal O_b(E)$, and we have $x = \bigwedge_\alpha x_\alpha$.
\end{lemma}

\begin{proof}
Assume $y \preceq_s x_\alpha$, $\alpha \in A$. We set $x_\alpha(t)=x_\alpha((-\infty,t])$ and $y(t)=y((-\infty,t])$. Then $x_\alpha(t)\le y(t)$ for each $t \in \mathbb R$ and each $\alpha \in A$. Define $x^0(t)=\bigvee_\alpha x_\alpha(t)$. Then $\{x^0(t): t \in \mathbb R\}$ satisfies the conditions of Proposition \ref{pr:2.4}.

Let $\{x(t): t \in \mathbb R\}$ be the the right-regularization of $\{x^0(t): t \in \mathbb R\}$, i.e.
$$
x(t)=x^0_r(t)= \bigwedge_{u>t} \bigvee_\alpha x_\alpha(u),\  t\in \mathbb R,
$$
see (3.9). By Proposition \ref{pr:2.4}, the system $\{x(t): t \in \mathbb R\}$ satisfies (4.2)--(4.4) of Theorem \ref{th:3.1}, so it defines a unique bounded observable $x$ on $E$ such that $x((-\infty,t])=x(t)$, $t \in \mathbb R$. Then $x(t)= \bigwedge_{u>t}\bigvee_\alpha x_\alpha(u) \ge \bigwedge_{u>t} x_\alpha(u) = x_\alpha(t)$ for each $\alpha \in A$, i.e. $x \preceq_s x_\alpha$ for each $\alpha \in A$. Now let $z$ be a bounded observable on $E$ such that $z \preceq_s x_\alpha$, $\alpha \in A$. Whence $x_\alpha(u)\le z(u)$, and $\bigvee_\alpha x_\alpha(u)\le z(u)$ which gets $x(t)= \bigwedge_{u>t}\bigvee_\alpha x_\alpha(u) \le \bigwedge_{u>t} z(u) = z(t)$, i.e. $z \preceq_s x$. Finally, $x = \bigwedge_\alpha x_\alpha$.
\end{proof}

\begin{lemma}\label{le:3.4}
Let $\{x_\alpha: \alpha \in A\}$ be a system of bounded observables of a complete lattice effect algebra $E$ such that there is a bounded observable $y$ on $E$ which is an upper bound of $\{x_\alpha: \alpha \in A\}$. Define
$$
x(t):= \bigwedge_\alpha x_\alpha((-\infty,t]),\ t \in \mathbb R. \eqno(4.6)
$$
Then the system $\{x(t): t \in \mathbb R\}$ satisfies conditions {\rm (3.2)--(3.4)} and it determines a unique bounded observable $x$ on $E$ and this observable is the least upper bound of $\{x_\alpha: \alpha \in A\}$ under the Olson order $\preceq_s$ in $\mathcal O_b(E)$, and we have $x = \bigvee_\alpha x_\alpha$.
\end{lemma}

\begin{proof}
Define $x(t)$ by (4.6). It is easy to see that $\{x(t): t \in \mathbb R\}$ satisfies (4.2)--(4.4).

Let $x$ be a unique observable of $E$ such that $x(t)=x((-\infty,t])$, $t \in \mathbb R$.

From the construction of $x$ we have $x(t)\le x_\alpha(t)$ for each $\alpha\in A$ and each $t\in \mathbb R$. This implies $x_\alpha\preceq_s x$ for each $\alpha \in A$. Since $x(t) \le y(t)$, there is $t_0\in \mathbb R$ such that $x(t)=1$ for each $t\ge t_0$, and $x(t)\le x_\alpha(t)$, there is $s_0\in \mathbb R$ such that $x(t)=0$ for $t<s_0$, showing $x$ is a bounded observable.

Finally, let $z$ be a bounded observable on $E$ such that $x_\alpha\preceq_s z$ for each $\alpha \in A$. Then $x_\alpha(t)\ge z(t)$ for each $\alpha\in A$, i.e. $x(t)\ge z(t)$, $t \in \mathbb R$. Hence, $x\preceq_s z$.
\end{proof}

As in the previous section, Lemmas \ref{le:3.3}--\ref{le:3.4} hold also for $\sigma$-complete effect algebras if $A$ is a countable set.

Let $\mathcal P(H)$ be the set of all orthogonal projections on a  real, complex or quaternionic Hilbert space $H$. It is an orthomodular complete lattice, see \cite{Var}. It is isomorphic to the system $\mathcal L(H)$ of all closed subspaces $M$ of a Hilbert space $H$. The isomorphism is given by $M \leftrightarrow P_M$, where $P_M$ is the orthogonal projections from $H$ onto $M\in \mathcal L(H)$. We recall that $\bigwedge_i M_i = \bigcap_i M_i$, $\bigvee_i M_i$ is the closed subspace of $H$ generated by $\bigcup_iM_i$, and $M'=M^\perp:=\{x \in H: x \perp y$ for each $y \in M\}$. In $\mathcal P(M)$, $P'=I-P$, where $I$ is the identity operator in $H$. The spaces $\mathcal P(H)$ and $\mathcal L(H)$ are complete lattice effect algebras. We note that any observable $x$ on $\mathcal P(H)$ is in fact a spectral measure and vice versa, so it determines a unique Hermitian operator (if $x$ is a bounded observable) or a self-adjoint operator (if $x$ is not a bounded observable). For bounded observables $x$ we have a unique Hermitian operator $A$ with the spectral measure $E_A$ (= observable) such that $A = \int t \dx E_A(t)$, see e.g. \cite{Hal}.

Applying the Olson order $\preceq_s$ on the set of all bounded observables of the complete lattice effect algebra  $\mathcal P(H)$, we see that it coincides with the spectral order of Hermitian operators as it was defined by Olson \cite{Ols}. In addition, $\preceq_s$ extends the standard ordering of projectors, is coarser than the usual ordering of operators. The same is true if we deal with projectors on a von Neumann algebra, \cite{Ols}.
So we have showed that we can introduce a partial order on many complete lattice effect algebras and  $\sigma$-complete effect algebras like $\sigma$-complete MV-algebras, $\mathcal P(H)$, $\sigma$-complete orthomodular lattices, etc. which extends original Olson's approach.

We note that an observable $x$ is {\it sharp} if $x(E)\in \Sh(E)$ for each Borel set $E \in \mathcal B(\mathbb R)$. For example, if $E=\mathcal P(H)$, then every observable is sharp. On the other hand, if $E$ is a $\sigma$-complete MV-algebra, then the question observable $q_a$ of an element $a \in E$ is sharp iff $a$ is sharp.

In addition, a lattice $(E,\le)$ is said to be an {\it involution lattice} if there is a mapping $^-: E \to E$ (called {\it involution} or {\it negation}) such that (i) $a^=:=a^{--}=a$ (double negation), $a\in E$, and (ii) $a\le b$ implies $b^-\le a^-$, $a,b \in E$ (contraposition).

Now we show that in the set $\mathcal{EO}(E)$ of a $\sigma$-lattice effect algebra of observables whose spectra are in the interval $[0,1]$ we can introduce a kind of negation which in the case of $\mathcal P(H)$ corresponds to the negation of operators, so that $\mathcal{EO}(E)$ will resemble the properties of $\mathcal E(H)$ with the spectral ordering.

\begin{theorem}\label{th:3.5}
Let $x\in\mathcal {EO}(E)$, where $E$ is a $\sigma$-lattice effect algebra (complete effect algebra). We define $x^-:=f(x)$, where $f(t)=1-t$. Then, for $x,y, x_\alpha \in \mathcal{EO}(E)$, we have
\begin{itemize}
\item[{\rm (i)}] $y^-\preceq_s x^-$ if $x\preceq_s y$.
\item[{\rm (ii)}] $x^{--}=x$.
\item[{\rm (iii)}] $x_0^-=x_1$, $x^-_1=x_0$.
\item[{\rm (iv)}] $(\bigwedge_\alpha x_\alpha)^- = \bigvee_\alpha x_\alpha^-$ and $(\bigvee_\alpha x_\alpha)^- = \bigwedge_\alpha x_\alpha^-$ if  the index set $A$ is countable ($A$ is an arbitrary set).
\item[{\rm (v)}] $q_a^- = q_{a'}$, $a \in E$.
\item[{\rm (vi)}] $q_a\wedge q_b= q_{a\wedge b}$, $q_a\vee q_b= q_{a\vee b}$, $a,b \in E$.
\item[{\rm (vii)}] Let $g(t)= \min\{t,1-t\}$ and $h(t)=\max\{1,1-t\}$. Then $x\wedge x^-\preceq_s g(x)$ and $h(x)\preceq_s x\vee x^-$.
\item[{\rm (viii)}] $q_a$ is a sharp observable $\Leftrightarrow$ $a$ is a sharp element $\Leftrightarrow$ $q_a\wedge q_a^- = q_0$.
\end{itemize}
In addition, $\mathcal{EO}(E)$ is an involution complete lattice with respect to the Olson order $\preceq_s$.
\end{theorem}

\begin{proof}
(i) We have $x\preceq_s y$ iff $y((-\infty,t))\le x((-\infty, t))$, $t \in \mathbb R$. By Lemma \ref{le:3.2}, this is equivalent to the condition $y((-\infty,t])\le x((-\infty, t])$, $t \in \mathbb R$.  Then $x((t,\infty))= 1-x((-\infty,t])\le 1-y((-\infty, t])= y((t,\infty))$. On the other side, $f(x)((-\infty,t))= x(\{s\in \mathbb R: 1-s<t\})= x(\{s\in \mathbb R: 1-t<s\}) = x((1-t,\infty))\le y((1-t,\infty))=f(y)((-\infty, t))$, that is $y^-\preceq_s x^-$.

(ii) $x^{--}=f(f(x)) = x$ while $f\circ f$ is the identity.

(iii) Use (3.5).

(iv) It follows simple lattice properties of $\mathcal{EO}(E)$ and of the negation $x^-$.

(v), (vi) Use (3.5).

(vii) Let $g(t)=\min\{t,1-t\}$. Then

$$
g(x)((-\infty, t))= \left\{\begin{array}{ll}
0 & \mbox{if} \ t\le 0,\\
x((-\infty,t)) & \mbox{if}\ 0< t\le 1/2,\\
x((1-t,\infty)) & \mbox{if}\ 1/2< t\le 1,\\
1 & \mbox{if}\ 1<t,
\end{array}
\right.
$$
and
$$
(x\wedge x^-)((-\infty, t))= \left\{\begin{array}{ll}
0 & \mbox{if} \ t\le 0,\\
x((-\infty,t))\vee x((1-t,\infty)) & \mbox{if}\ 0< t\le 1,\\
1 & \mbox{if}\ 1<t,
\end{array}
\right.
$$
which shows $g(x)((-\infty,t))\le (x\wedge x^-)((-\infty,t))$, i.e. $x\wedge x^-\preceq_s g(x)$.

The second property follows from (i), (iv) and the latter proved property.

(viii) It is evident that $q_a$ is sharp iff $a$ is sharp. Now let $a$ be sharp. Then due to (v)--(vi), we have $q_a\wedge q_a^-= q_{a\wedge a'}= q_0$, and vice versa.

The final statement follows from Theorem \ref{th:2.6}.
\end{proof}

\begin{example}\label{ex:4.6}
If $x$ is a sharp observable, then not necessarily $x\wedge x^-=q_0$. Indeed, let $E=\mathcal B(\mathbb R)$, $h(t) = 0$, if $t\le 0$, $h(t)=t$ if $0\le t \le 1$, $h(t)=1$ if $t>1$, and $h_1(t) = 1$ if $t\le 1$, $h_1(t)=1-t$, $h_1(t)=0$ if $t>1$. Let $x=h^{-1}$. Then $x^-= f(x)=h^{-1}\circ f^{-1}=h_1^{-1}$. Hence, if $t=0.3$, then $x((-\infty,t))=(-\infty, 0.3)$, $x^-((-\infty,0.3))=((0.3,\infty))$. So $x((-\infty,t)) \vee x^-((-\infty,t)) = x((-\infty,t)) \cup x^-((-\infty,t))= (-\infty, 0.3)\cup (0.3,\infty)\ne \mathbb R=1$.
\end{example}

\section{Miscellaneous Properties of the Olson Order}

Now we show that the Olson order on $\mathcal O_b(E)$ in the case $E$ is a $\sigma$-algebra of subsets generalizes also the natural order of functions.

We remind that if $f$ and $g$ are two real-valued functions on $\Omega$, then we write $f\leqslant g$ iff $f(\omega)\le g(\omega)$ for each $\omega \in \Omega$.

\begin{theorem}\label{th:5.1}
Let $E=\mathcal S$, where $\mathcal S$ is a $\sigma$-algebra of subsets of a non-void set $\Omega$. If $f:\Omega \to \mathbb R$ is an $\mathcal S$-measurable function, then the mapping $x_f:=f$, i.e. $x_f(E)=f^{-1}(E)$, $E \in \mathcal B(\mathbb R)$ is an observable. Conversely, for every observable $x$, there is a unique $\mathcal S$-measurable function $f:\Omega \to \mathbb R$ such that $x= x_f$. Therefore, for every observable $x$ we have $x(\bigcup_n E_n)=\bigcup_n x(E_n)$ for any sequence $\{E_n\}$ of Borel sets.

Moreover, $x_f$ is bounded if and only if $f$ is bounded.

Let $f,g$ be two $\mathcal S$-measurable function. Then $x_f \preceq_s x_g$ if and only if $f\leqslant g$.

In addition, let $\{f_n\}$ and $\{g_n\}$ be sequences of bounded $\mathcal S$-measurable functions on $\Omega$ such that there are bounded $\mathcal S$-measurable functions $f_0$ and $g_0$ with $f_0\leqslant f_n$ and $g_n \leqslant g_0$ for each positive integer $n$. Then
$$
\bigwedge_n x_{f_n} = x_f,\quad \bigvee_n x_{g_n} = x_g, \eqno(5.1)
$$
where $f= \inf_n f_n$ and $g=\sup_n g_n$.
\end{theorem}

\begin{proof}
First we show that if $x$ is an observable on $E$, then there is a unique $\mathcal S$-measurable function $f$ such that $x= x_f$.

{\it Existence of $f$.} Let $r_1,r_2,\ldots$ be any enumeration of the set of rational numbers. If we set
$$
f(\omega)=\inf\{r_j: \omega \in x((-\infty,r_j))\},
$$
then $f$ is a well-defined finite function for which we have
$$
f^{-1}((-\infty,r_k))=\bigcup_{i: r_i< r_k} x((\infty,r_i)),
$$
hence, $f$ is $\mathcal S$-measurable such that $x((-\infty,r_k))=x_f((-\infty,r_k))$ for each $r_k$. If $t$ is any real number, there is a sequence $\{s_n\}$ of rational numbers such that $\{s_n\}\nearrow t$, therefore, $x((-\infty,t))=x_f((-\infty,t))$ for each $t \in \mathbb R$. Consequently, if $\mathcal K=\{E \in\mathcal B(\mathbb R): x(E)=x_f(E)\}$, then $(-\infty,t)\in \mathcal K$, $t\in \mathbb R$ as well as $[a,b)\in \mathcal K$ for $a\le b$. The family $\mathcal K$ is a Dynkin system and similarly as in the proof of Theorem \ref{th:3.1}, we can conclude $\mathcal K=\mathcal B(\mathbb R)$. Whence, $x=x_f$ which in particular means that $x$ is a $\sigma$-homomorphism from $\mathcal B(\mathbb R)$ into $\mathcal S$ preserving countable unions of Borel sets.

{\it Uniqueness of $f$.} Let $g$ be an arbitrary finite $\mathcal S$-measurable function on $\Omega$ such that $x=x_g$ and assume that $f\ne g$. Then there is $\omega_0 \in \Omega$ such that $f(\omega_0)>g(\omega_0)$ or $f(\omega_0) < g(\omega_0)$. In the first case, we choose a real number $t$ such that $g(\omega_0)<t<f(\omega_0)$. Then $\omega_0 \in \{\omega: g(\omega)<t\}=\{\omega: f(\omega_0)<t\}$. From the choice of $\omega_0$, we have $\omega_0 \not\in \{\omega: f(\omega_0)<t\}$, an absurd. Similarly for the second case. Whence, $f=g$.

Now let $f\leqslant g$, then $x_g((-\infty,t))=\{\omega: g(\omega)<t\}\subseteq \{\omega: f(\omega) <t\}$, i.e. $x_f \preceq_s x_g$.  Conversely, let $x_f \preceq_s x_g$, and assume that there is $\omega_0\in \Omega$ such that $f(\omega_0) > g(\omega_0)$. Choose $t \in \mathbb R$ with $g(\omega_0)<t<f(\omega_0)$. Then $\omega_0\in \{\omega: g(\omega)<t\} \subseteq \{\omega: f(\omega)<t\}$, but $\omega_0\not\in \{\omega: f(\omega)<t\}$, which is a contradiction, so that $f\leqslant g$.

Equalities in (5.1) follow from the first part of the proof, and from the fact that $\inf_nf_n$ and $\sup_ng_n$ are $\mathcal S$-measurable functions.
\end{proof}

The just proved Theorem  can be generalized as follows.

Let $\mathcal T$ be a tribe of $[0,1]$-functions on $\Omega\ne \emptyset$.
Motivating by \cite{JPV}, we say that mapping $K: \Omega \times \mathcal B(\mathbb R)\to [0,1]$ is a {\it Markov kernel associated with} $\mathcal T$ (simply {\it Markov kernel}) if
\begin{itemize}
\item[(i)] for any fixed $E \in \mathcal B(\mathbb R)$, the mapping $K(\cdot,E)\in \mathcal T$;
\item[(ii)] for any fixed $\omega \in \Omega$, the mapping $K(\omega,\cdot)$ is a probability measure on $(\mathbb R, \mathcal B(\mathbb R))$.
\end{itemize}

\begin{theorem}\label{th:5.3}
Let $\mathcal T$ be a tribe of functions on a non-empty set $\Omega$. Let $K$ be a Markov kernel associated with $\mathcal T$. Then the mapping $x_K:\mathcal B(\mathbb R)\to \mathcal T$ defined by $x_K(E):= K(\cdot,E)$, $E\in \mathcal B(\mathbb R)$, is an observable on $\mathcal T$. Conversely, let $x$ be an observable on $\mathcal T$. Then there is a unique Markov kernel $K$ associated with $\mathcal T$ such that $x=x_K$.

Let $K$ and $H$ be two Markov kernels. We write $K \preceq H$ if $H(\omega,(-\infty,t))\le K(\omega,(-\infty,t))$, $\omega\in \Omega$, $t \in \mathbb R$.  Then $x_K\preceq_s x_H$ if and only if $K \preceq H$.

Finally, if $\{x_n\}$ is a countable system of bounded observables bounded from below by a bounded observable, then for $x=\bigwedge_n x_n$, we have
$$
x((-\infty,t))(\omega)= \sup_n x_n((-\infty, t)(\omega),\ \omega \in \Omega,\ t \in \mathbb R.
$$
Similarly, if $\{y_n\}$ is a countable system of bounded observables bounded from above by a bounded observable, then for $y = \bigvee_n y_n$, we have
$$
y((-\infty,t))(\omega)= \sup_{u_n\nearrow t}\inf_n x_n((-\infty,u_n))(\omega), \ \omega \in \Omega,\ t \in \mathbb R.
$$
\end{theorem}

\begin{proof}
Let $K$ be a Markov kernel and let $x_K(E):= K(\cdot,E)$, $E \in \mathcal B(\mathbb R)$. Then $x_K(\emptyset) = 0_\Omega$ and $x_K(\mathbb R)=1_\Omega$, where $0_\Omega(\omega)=0$ and $1_\Omega(\omega)=1$, $\omega \in \Omega$. If $E\cap F=\emptyset$, then $x_K(E\cup F)=K(\cdot, E\cup F)= K(\cdot, E)+K(\cdot,F)=x_K(E)+x_K(F)$. Similarly, if $E_n\nearrow E$, then $x_K(E_n)\nearrow x_K(E)$, proving $x_K$ is an observable on $\mathcal T$.

Now let $x$ be an observable on a tribe $\mathcal T$.

{\it Existence of $K$.}
We denote by $b_t=x((-\infty,t))\in \mathcal T$ for each $t \in \mathbb R$. Then $\{b_t: t \in \mathbb R\}$ is a system of functions from $\mathcal T$ such that $b_s\le b_t$ whenever $s<t$.
Let $\omega$ be a fixed element of $\Omega$. We define $F_\omega(t):= b_t(\omega)$, $t \in \mathbb R$.  Due to (3.2)--(3.4), we see that $F_\omega$ is a non-decreasing, left continuous function, such that $\lim_{t \to -\infty} F_\omega(t)=0$  and $\lim_{t \to \infty} F_\omega(t)=1$. By \cite[Thm 43.2]{Hal}, $F_\omega$ is a distribution function on $\mathbb R$ corresponding to the unique probability measure $P_\omega$ on $\mathcal B(\mathbb R)$, that is, $P_\omega((-\infty,t))=F_\omega(t)$ for every $t \in \mathbb R$. Define now a mapping $K:\Omega \times \mathcal B(\mathbb R) \to [0,1]$ by $K(\omega,E)=P_\omega(E)$, $\omega \in \Omega$ and $E \in \mathcal B(\mathbb R)$.   In particular, we have $K(\cdot,(-\infty,t)) = b_t\in \mathcal T$ for any $t \in \mathbb R$.

We assert that $K$ is a Markov kernel. First we show that every $K(\cdot,E) \in \mathcal T$ for any $E \in \mathcal B(\mathbb R)$, let $\mathcal K$ be the system of all $E \in \mathcal B(\mathbb R)$ such that $K(\cdot,E) \in \mathcal T$. Then $\mathcal K$ is a Dynkin system. The system $\mathcal K$ contains all intervals $(-\infty, t)$ for $t \in \mathbb R$, which is a $\pi$-system.  Similarly as in the proof of Theorem \ref{th:3.1}, $\mathcal K$ is a $\sigma$-algebra, and hence $\mathcal K = \mathcal B(\mathbb R)$. Since $K(\omega,E)=P_\omega(E)$, we see that $K$ is a Markov kernel. Now if $\mathcal H=\{E \in \mathcal B(\mathbb R): x(E)=x_K(E)\}$, in the same way as for $\mathcal K$, we can show $\mathcal H=\mathcal B(\mathbb R)$.

{\it Uniqueness of $K$.} Let $H$ be another Markov kernel such that $x=x_H$. Then $x((-\infty,t))(\omega)= P_\omega(E)= K(\omega,E)=H(\omega, E)$, $\omega \in \Omega$ and $E \in \mathcal B(\mathbb R)$. Similarly as for $\mathcal K$ and $\mathcal H$, we can prove $K=H$.

The property $x_K\preceq_s x_H$ iff $K\preceq H$ follows from the definitions of $\preceq_s$ and $\preceq$.

The formulas for infima and suprema follow from definition of $\preceq_s$, (3.7) and (3.10).
\end{proof}

Now if $\mathcal T$ is an effect-tribe of $[0,1]$-valued functions on  $\Omega \ne \emptyset$, we can define a {\it Markov kernel} associated with the effect-tribe $\mathcal T$ in the same way as that for a tribe. In the same way as Theorem \ref{th:5.3}, we can prove the following result.

\begin{theorem}\label{th:5.4}
Let $\mathcal T$ is an effect-tribe of functions on a non-empty set $\Omega$. Let $K$ be a Markov kernel associated with $\mathcal T$. Then the mapping $x_K:\mathcal B(\mathbb R)\to \mathcal T$ defined by $x_K(E):= K(\cdot,E)$, $E\in \mathcal B(\mathbb R)$, is an observable on $\mathcal T$. Conversely, let $x$ be an observable on $\mathcal T$. Then there is a unique Markov kernel $K$ associated with $\mathcal T$ such that $x=x_K$.

Let $K$ and $H$ be two Markov kernels. We write $K \preceq H$ if $H(\omega,(-\infty,t))\le K(\omega,(-\infty,t))$, $\omega\in \Omega$, $t \in \mathbb R$.  Then $x_K\preceq_s x_H$ if and only if $K \preceq H$.
\end{theorem}

\begin{theorem}\label{th:5.2}
Let $E$ be a Boolean $\sigma$-algebra, $\mathcal S$ a $\sigma$-algebra of subsets of a non-void set, and $h$ a $\sigma$-homomorphism from $\mathcal S$ onto $E$.

{\rm(1)} If $x$ is an observable on $E$, there is an $\mathcal S$-measurable function $f: \Omega\to \mathbb R$ such that $x= h\circ f^{-1}$. If in addition, $x=h\circ h_1^{-1}$, then $h(\{\omega: f(\omega)\ne g(\omega)\}=0$.

If $x= h\circ f^{-1}$ and $y=h\circ g^{-1}$, then $x\preceq_s y$ if and only if $h(\{\omega: g(\omega) < f(\omega)\})=0$.

Let $x_n= h\circ f_n$, $n\ge 1$. If $\{x_n\}$ are bounded observables bounded from below, then $\bigwedge_n x_n= h\circ (\inf_n f_n)$. If $\{x_n\}$ are bounded bounded observables bounded from above, then $\bigvee_n x_n= x\circ (\sup_n f_n)$.

{\rm (2)} Let $\{x_n\}$ be a system of bounded observables of the Boolean algebra $E$. There are an observable $x$ and a sequence of Borel measurable functions $\{f_n\}$ from $\mathbb R$ into $\mathbb R$ such that $x_n=f_n(x)$, $n\ge 1$.

If $x=f(z)$ and $y=g(z)$, then $x\preceq_s$ if and only if $x(\{s: g(s)>f(s)\})=0.$ In addition, if $\{f_n(x)\}$ is bounded from below, then $\bigwedge_n f(x)=(\inf_nf_n)(x)$ and if $\{f_n(x)\}$ is bounded from above, then $\bigvee_n f_n(x)=(\sup_n f_n)(x)$ under the Olson order.
\end{theorem}

\begin{proof}
(1) If $x$ is an observable, applying \cite[Thm 1.4]{Var}, we can find a real-valued  $\mathcal S$-measurable function $f$ on $\Omega$ such that $x=h\circ f^{-1}$. Uniqueness of $f$ in the mentioned sense was also guaranteed in \cite[Thm 1.4]{Var}.

Let $x=h\circ f^{-1}$ and $y=h\circ g^{-1}$. Assume $x\preceq_s y$. Then $y((-\infty,t))\le x((-\infty,t))$ for each real $t$. Calculate

\begin{eqnarray*}
h(\{\omega: g(\omega)<f(\omega)\})&= &h(\bigcup_{r \in \mathbb Q}\{\omega: g(\omega)<r< f(\omega)\})\\
&=& \bigvee_{r \in \mathbb Q}h(\{\omega: g(\omega) <r\} \cap \{\omega: r< f(\omega)\})\\
&=&\bigvee_{r \in \mathbb Q} y((-\infty,r))\wedge x([r,\infty))\\
&\le & \bigvee_{r \in \mathbb Q} x((-\infty,r))\wedge x([r,\infty))=0.
\end{eqnarray*}

Conversely, assume $h(\{\omega: g(\omega)<f(\omega)\})=0$. Then we have $y((-\infty,t)) = h(g^{-1}((-\infty, t)))= h(\{\omega: g(\omega)< t\}\cap \{\omega: g(\omega)<f(\omega)\}) \vee h(\{\omega: g(\omega)< t\}\cap \{\omega: f(\omega)\le g(\omega)\}) = h(\{\omega: g(\omega)< t\}\cap \{\omega: f(\omega)\le g(\omega)\}) \le h(\{\omega: f(\omega)< t\})= x((-\infty,t))$.

(2) The first part follows from \cite[Thm 6.9]{Var}.

The rest follows the same steps as the proof of (1).
\end{proof}

\section{Conclusion}

The set $\mathcal E(H)$ of effect operators on a Hilbert space $H$ is not a lattice under the standard ordering of operators, whilst the set $\mathcal P(H)$ of all orthogonal projections is a complete orthomodular lattice. Olson \cite{Ols} has introduced a new ordering of operators of $\mathcal E(H)$, called spectral order, such that it coincides on projectors with the standard one and $\mathcal E(H)$ becomes a complete lattice.

Inspiring by Olson \cite{Ols}, we have introduced a partial order, called the Olson order, for the set of bounded observables of a monotone $\sigma$-complete effect algebra $E$. If $E$ is a complete effect algebra, in Lemmas \ref{le:2.3} and \ref{le:2.5}, we have showed that infimas and supremas of any system of bounded observables under the Olson order exists and we showed how to calculate them. Therefore, as it is stated in Theorem \ref{th:2.6}, the set of bounded observables on $E$ is a Dedekind complete lattice. In addition,  the set of observables whose spectra lie in the real interval $[0,1]$ can be equipped with  negation, so that it resembles the properties of $\mathcal E(H)$, Theorem \ref{th:3.5}. In the case $E$ is only a $\sigma$-lattice effect algebra, the set of bounded observables is a Dedekind $\sigma$-lattice. We present also an equivalent definition of the Olson order that is more close to the original one defined by Olson for $\mathcal E(H)$.

Finally, we show how the Olson order can be organized for $\sigma$-algebras of subsets, tribes ($\sigma$-complete MV-algebras of functions with pointwise defined operations), effect-tribes (monotone $\sigma$-complete effect algebras of functions with pointwise defined operations). Some illustrating examples are provided.

The paper provides a new tool for studying observables using lattice properties of the set of bounded observables under the Olson order.

\end{document}